\documentclass[%
  a4paper,%
  11pt,%
  abstracton,%
  smallheadings,%
]{scrartcl}

\usepackage[english]{babel}
\usepackage[T1]{fontenc}
\usepackage[utf8]{inputenc}
\usepackage{textcomp}

\usepackage{mathptmx}
\usepackage{courier}

\usepackage{amsmath,amssymb,amsthm}
\usepackage{graphicx}
\usepackage{titlesec}
\usepackage{microtype}
\usepackage[numbers,sort]{natbib}

\renewcommand*{\vec}[1]{\mathbf{#1}}
\newcommand*{\mat}[1]{\mathbf{#1}}
\newcommand*{\RSet}{\mathbb{R}}

\newcommand*{\defeq}{\mathrel{\mathop:}=}

\newcommand*{\SQ}{\mathcal{S}}
\newcommand*{\tp}{T}
\DeclareMathOperator*{\rank}{rank}

\DeclareMathOperator*{\dist}{dist}

\setkomafont{sectioning}{\normalfont\bfseries}

\newtheorem{theorem}{Theorem}

\newtheorem{lemma}[theorem]{Lemma}

\title{Contributions to Four-Position Theory\\with Relative Rotations}
\author{Hans-Peter Schröcker\thanks{Hans-Peter Schröcker, Unit Geometry and CAD, University
  Innsbruck, Technikerstraße~13, A6020~Innsbruck, Austria,
  \url{http://geometrie.uibk.ac.at/schroecker/}}}
\date{\today}

\begin{document}

\maketitle
\begin{abstract}
We consider the geometry of four spatial displacements, arranged in
cyclic order, such that the relative motion between neighbouring
displacements is a pure rotation. We compute the locus of points whose
homologous images lie on a circle, the locus of oriented planes whose
homologous images are tangent to a cone of revolution, and the locus
of oriented lines whose homologous images form a skew quadrilateral on
a hyperboloid of revolution.


\end{abstract}

{\noindent\small
}

\section{Introduction}
\label{sec:introduction}

Four-positions theory studies the geometry of four proper Euclidean
displacements, see
\cite[Chapter~5,~\S8--9]{bottema90:_theoretical_kinematics}. Often, it
is considered as a discretization of differential properties of order
three of a smooth one-parameter motion. In this article we take a
different point of view and regard the four displacements as an
elementary cell in a quadrilateral net of positions, that is, the
discretization of a two- or more-parameter motion. Motivated by
applications in the rising field of discrete differential geometry
\cite{bobenko08:_discrete_differential_geometry} we intend to study
quadrilateral nets of positions such that the relative displacement
between neighbouring positions is a pure rotation. The present text
shall provide foundations for this project. Thus, the questions we are
interested in are induced from discrete differential
geometry. Nonetheless, our contributions seem to be of a certain
interest in their own right.

Denote the four displacements by $\alpha_0$, $\alpha_1$, $\alpha_2$,
$\alpha_3$ and the relative displacements of consecutive positions by
$\tau_{i,i+1} = \alpha_{i+1} \circ \alpha_i^{-1}$ (indices modulo
four). We require that every relative displacement $\tau_{i,i+1}$ is a
\emph{pure rotation.} In this case we call the quadruple $(\alpha_0,
\alpha_1, \alpha_2, \alpha_3)$ a \emph{rotation quadrangle.}  Rotation
quadrangles with additional properties (the displacements $\alpha_2
\circ \alpha_0^{-1}$ and $\alpha_3 \circ \alpha_1^{-1}$ are also
rotations) have been studied in
\cite[Chapter~5,~\S~9]{bottema90:_theoretical_kinematics}.  Our
contribution is more general and we provide answers to questions not
asked in \cite{bottema90:_theoretical_kinematics}.

In Section~\ref{sec:construction-rotation-quadrilaterals} we address
the problem of constructing rotation quadrilaterals. Study's kinematic
mapping turns out to be a versatile tools for this task. Subsequently,
we study points $X$ whose homologous images $X_i = \alpha_i(X)$ lie on
a circle (Section~\ref{sec:homologous-points};
Figure~\ref{fig:homologous-points-planes-lines}, left), oriented
planes $\varepsilon$ whose homologous images $\varepsilon_i =
\alpha_i(\varepsilon)$ are tangent to a cone of revolution
(Section~\ref{sec:homologous-planes};
Figure~\ref{fig:homologous-points-planes-lines}, center), and oriented
lines $\ell$ whose homologous images $\ell_i = \alpha_i(\ell)$ form a
skew quadrilateral on a hyperboloid of revolution
(Section~\ref{sec:homologous-lines};
Figure~\ref{fig:homologous-points-planes-lines}, right).

\begin{figure}
  \centering
  \includegraphics{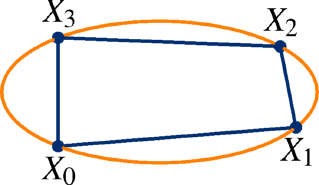}\quad
  \includegraphics{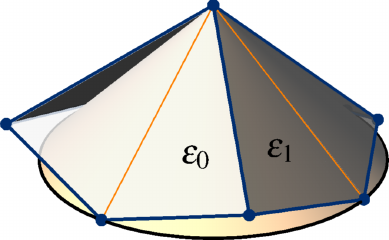}\quad
  \includegraphics{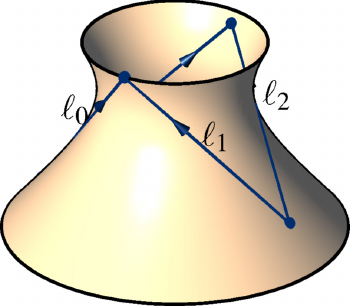}
  \caption{Homologous points on a circle, homologous planes tangent to
    a cone of revolution, and homologous lines on a hyperboloid of
    revolution}
  \label{fig:homologous-points-planes-lines}
\end{figure}

\section{Construction of rotation quadrilaterals}
\label{sec:construction-rotation-quadrilaterals}

One position in a rotation quadrilateral, say $\alpha_0$, can be
chosen arbitrarily without affecting the quadrilateral's geometry. The
remaining positions can be constructed inductively by suitable
rotations but the final position $\alpha_3$ must be such that both
$\tau_{23}$ and $\tau_{30}$ are rotations. This is certainly possible,
albeit a little tedious, in Euclidean three space. The quadric model
of Euclidean displacements, obtained by Study's kinematic mapping
\cite{husty09:_algebraic_geometry_kinematics,selig05:_fundamentals_robotics}
seems to be a more appropriate setting for this task.  It transforms
two positions corresponding in a relative rotation to conjugate points
on the Study quadric $\SQ$. More precisely, the composition of all
rotations about a fixed axis with a fixed displacement is a straight
line in the Study quadric. Thus, a rotation quadrilateral corresponds,
via Study's kinematic mapping, to a skew quadrilateral $Q$ in the
Study quadric. We require here and in the following that $Q$ spans a
three-space which is not completely contained in the Study quadric
$\SQ$. This genericity assumption is not fulfilled in the special
configuration studied in
\cite[Chapter~5,~\S~9]{bottema90:_theoretical_kinematics}.

The quadrilateral $Q$ is uniquely determined by three suitable chosen
vertices $A_0$, $A_1$, $A_2$ and the connecting line of $A_2$ and the
missing vertex $A_3$. Alternatively, $Q$ is determined by two opposite
vertices, say $A_0$ and $A_2$, and two opposite edges. In any case the
input data defines a three-space whose intersection with the Study
quadric is a hyperboloid $H$ containing $Q$. The completion of the
missing data is easy. Denoting the relative revolute axes in the
moving space by $r_{i,i+1}$ the translation of these considerations
into the language of kinematics reads:

\begin{theorem}
  \label{th:rotation-quadrangle-1}
  A rotation quadrangle is uniquely determined by
  \begin{enumerate}
  \item a fixed displacement $\alpha_i$, three revolute axis
    $r_{i,i+1}$, $r_{i+1,i+2}$, $r_{i+2,i+3}$, and two rotation
    angles $\omega_{i,i+1}$ and~$\omega_{i+1,i+2}$ (indices modulo
    four) or
  \item by two displacement $\alpha_i$, $\alpha_{i+2}$ and two
    relative revolute axes $r_{i,i+1}$, $r_{i+2,i+3}$ (indices modulo
    four).
  \end{enumerate}
\end{theorem}

Computing the actual rotation angles and line coordinates of the
revolute axes from the quadrilateral $Q$ is straightforward. The
necessary formulas are, for example, found in
\cite[Satz~13]{weiss35:_liniengeometrie_und_kinematik}. Beware that
\cite{weiss35:_liniengeometrie_und_kinematik} uses conventions for the
Study parameters that slightly differ from modern authors, for example
\cite{husty09:_algebraic_geometry_kinematics} or
\cite{selig05:_fundamentals_robotics}. Therefore minor adaptions to
the formulas may be necessary.

\section{Homologous points on a circle}
\label{sec:homologous-points}

Now we turn to the study of the locus of points in the moving space
whose homologous images lie on a circle
(Figure~\ref{fig:homologous-points-planes-lines}, left). As usual we
consider straight lines as special circles of infinite radius. If the
four relative displacements are general screws, the sought locus is
known to be an algebraic curve of degree six
\cite[p.~128]{bottema90:_theoretical_kinematics}. In case of rotation
quadrilaterals this is still true but the curve is highly reducible:

\begin{theorem}
  \label{th:homologous-points}
  The locus of points in the moving space whose homologous images lie
  on a circle consists of the four relative revolute axis $r_{01}$,
  $r_{12}$, $r_{23}$, $r_{30}$ and their two transversals $u$ and~$v$.
\end{theorem}

\begin{proof}
  By the general result of \cite{bottema90:_theoretical_kinematics}
  the sought locus can consist of six lines at most. We show that all
  points on the six lines mentioned in the theorem have indeed
  homologous images on circles. For points $X \in r_{i,i+1}$ this is
  trivially true since we have $X_i = X_{i+1}$. Since any transversal
  carries four points of this type, Lemma~\ref{lem:coplanar-points}
  below, implies the same property for all points on $u$ and~$v$. (At
  least the first part of this lemma is well-known. An old reference
  is \cite[p.~187]{schoenflies86:_geometetrie_der_bewegung}.)
\end{proof}

\begin{lemma}
  \label{lem:coplanar-points}
  Consider a straight line $\ell$ in the moving space and four affine
  displacements. If $\ell$ contains four points whose homologous
  images are co-planar than this is the case for all points of
  $\ell$. If, in addition, the homologous images of three points lie
  on circle, this is the case for all points of~$\ell$.
\end{lemma}

\begin{proof}
  We can parametrize the homologous images of $\ell$ as
  \begin{equation}
    \label{eq:3}
    \ell_i\colon \vec{x}_i(t) = \vec{a}_i + t \vec{d}_i;\quad
    \vec{a}_i,\ \vec{d}_i \in \RSet^3;\quad i = 0,1,2,3
  \end{equation}
  such that points of the same parameter value $t$ are homologous.
  The co-planarity condition reads
  \begin{equation}
    \label{eq:4}
    \det
    \begin{pmatrix}
      1         & 1         & 1         & 1 \\
      \vec{x}_0 & \vec{x}_1 & \vec{x}_2 & \vec{x}_3
    \end{pmatrix} = 0
  \end{equation}
  (we omit the argument $t$ for sake of readability). Since it is at
  most of degree three in $t$, it vanishes identically if four zeros
  exist. This proves the first assertion of the lemma.

  Four homologous points $X_0$, $X_1$, $X_2$, $X_3$ lie on a circle
  (or a straight line), if three independent bisector planes
  $\beta_{ij}$ of $X_i$ and $X_j$ have a line, possibly at infinity,
  in common.  This is equivalent to
  \begin{equation}
    \label{eq:5}
    \rank
    \begin{pmatrix}
      \vec{x}_0^\tp \cdot \vec{x}_0 + \vec{x}_1^\tp \cdot \vec{x}_1 & \vec{x}_0 - \vec{x}_1 \\
      \vec{x}_1^\tp \cdot \vec{x}_1 + \vec{x}_2^\tp \cdot \vec{x}_2 & \vec{x}_1 - \vec{x}_2 \\
      \vec{x}_2^\tp \cdot \vec{x}_2 + \vec{x}_3^\tp \cdot \vec{x}_3 & \vec{x}_2 - \vec{x}_3
    \end{pmatrix} \le 2.
  \end{equation}
  Corresponding points are assumed to be coplanar. Therefore there
  exist scalars $\lambda$, $\mu$ and $\nu$, not all of them zero, such
  that
  \begin{equation}
    \label{eq:6}
    \lambda (\vec{x}_0 - \vec{x}_1) +
    \mu     (\vec{x}_1 - \vec{x}_2) +
    \nu     (\vec{x}_2 - \vec{x}_3) = 0.
  \end{equation}
  Hence, the circularity condition becomes
  \begin{equation}
    \label{eq:7}
    \lambda (\vec{x}_0^\tp \cdot \vec{x}_0 + \vec{x}_1^\tp \cdot \vec{x}_1) +
    \mu     (\vec{x}_1^\tp \cdot \vec{x}_1 + \vec{x}_2^\tp \cdot \vec{x}_2) +
    \nu     (\vec{x}_2^\tp \cdot \vec{x}_2 + \vec{x}_3^\tp \cdot \vec{x}_3) = 0.
  \end{equation}
  It is at most quadratic in $t$. Therefore, existence of three zeros
  implies its vanishing and the second assertion of the lemma is
  proved.
\end{proof}

The set of circles through quadruples of homologous points originating
from a relative revolute axis $r_{i,i+1}$ is not very interesting in
the context of this article. Of more relevance is the set of circles
generated by points of the transversal lines $u$ and~$v$.

\begin{theorem}
  \label{th:hyperboloid}
  The circles through homologous images of points of a transversal
  line ($u$ or $v$) lie on a hyperboloid of revolution.
\end{theorem}

\begin{figure}
  \centering
  \includegraphics{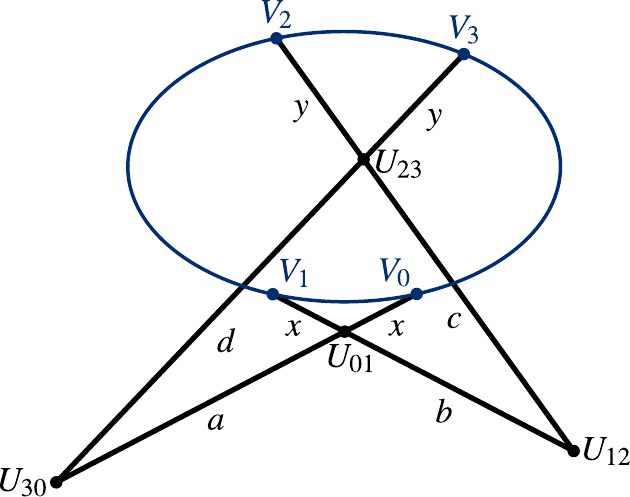}
  \caption{Homologous images of the transversal line $u$}
  \label{fig:trajectory-circles}
\end{figure}

\begin{proof}
  We consider only points on the transversal $u$.  Any two of its
  successive homologous images $u_i$, $u_{i+1}$ intersect in a point
  $U_{i,i+1}$, the intersection point of $\alpha_i(r_{i,i+1})$ and
  $\alpha_{i+1}(r_{i,i+1})$. Therefore the lines $u_0$, $u_1$, $u_2$,
  $u_3$ are the edges of a skew quadrilateral and lie on a pencil of
  hyperboloids. The intersection points with circles through
  homologous points form congruent ranges of points on the lines
  $u_0$, $u_1$, $u_2$, $u_3$ and the intersection points $U_{i,i+1}$
  correspond in this relation. Consider now a circle $C$ such that
  every intersections point $V_i$ with $u_i$ lies outside the segment
  $[U_{i-1,i},U_{i,i+1}]$ for $i \in \{0,1,2,3\}$. Setting
  \begin{equation}
    \label{eq:8}
    \begin{gathered}
      a \defeq \dist (U_{30},U_{01}),\quad
      b \defeq \dist (U_{01},U_{12}),\quad
      x \defeq \dist (U_{01}, V_0) = \dist (U_{01}, V_1)\\
      c \defeq \dist (U_{12},U_{23}),\quad
      d \defeq \dist (U_{30},U_{23}),\quad
      y \defeq \dist (U_{23}, V_2) = \dist (U_{23}, V_2)
    \end{gathered}
  \end{equation}
  (Figure~\ref{fig:trajectory-circles}) we see that $b + x = c + y$
  and $a + x = d + y$. This implies $a+c = b+d$. In other words, both
  pairs of opposite edges in the skew quadrilateral $(U_{01}, U_{12},
  U_{23}, U_{30})$ have the same sum of lengths. This is a well-known
  criterion for the existence of a hyperboloid of revolution through
  $u_0$, $u_1$, $u_2$, and~$u_3$.
\end{proof}

\section{Homologous planes tangent to a cone of revolution}
\label{sec:homologous-planes}

In this section we study the locus of planes $\varepsilon$ in the
moving space such that its homologous images $\varepsilon_0$,
$\varepsilon_1$, $\varepsilon_2$, $\varepsilon_3$ are tangent contact
to a cone of revolution
(Figure~\ref{fig:homologous-points-planes-lines}, center). Actually we
will not discuss this in full generality but restrict ourselves to the
case of \emph{oriented} planes in \emph{oriented} contact with a cone
of revolution. This means the cone axis is the intersection line of
the \emph{unique} bisector planes of $\varepsilon_i$ and
$\varepsilon_{i+1}$. Having in mind applications in discrete
differential geometry (a kinematic generation of conical nets, see
\cite{pottmann08:_focal_circular_conical}) this additional assumption
is justified.

Since the offset surface of a cone of revolution is again a cone of
revolution, every plane parallel to a solution plane $\varepsilon$ is
a solution plane as well. We infer that the sought locus of planes
consists of \emph{pencils of parallel planes}. This can also be
confirmed by direct computation. Denote by $\vec{e} =
(e_0,e_1,e_2,e_3)^\tp$ the plane coordinate vector of
$\varepsilon$. If $\mat{A}_i$ is the homogeneous transformation matrix
describing $\alpha_i$, the coordinate vectors $\vec{e}_i$ of the
homologous planes are found from
\begin{equation}
  \label{eq:10}
  \vec{e}_i^\tp = \vec{e}^{\tp} \cdot \mat{A}_i^{-1},
  \quad
  i \in \{0,1,2,3\}.
\end{equation}
The condition that they have a point in common reads
\begin{equation}
  \label{eq:11}
  E(e_0,e_1,e_2,e_3) = \det(\vec{e}_0, \vec{e}_1, \vec{e}_2, \vec{e}_3) = 0
\end{equation}
and is of degree four in $e_0$, $e_1$, $e_2$, and $e_3$. Assuming that
\eqref{eq:11} holds, the condition on the conical position can be
stated in terms of the oriented normal vectors $\vec{n}_i$ which are
obtained from $\vec{e}_i$ by dropping the first (homogenizing)
coordinate. The planes are tangent to a cone of revolution if the four
points with coordinate vectors $\vec{n}_i$ lie on a circle (note that
these vectors have the same length). This condition results in a
homogeneous cubic equation $G(e_1,e_2,e_3) = 0$ which is nothing but
the circle-cone condition known from four-position synthesis of
spherical four-bar linkages (see
\cite[p.~179]{mccarthy00:_geometric_design_linkages}).

A closer inspection shows that there exists a homogeneous polynomial
$F(e_1,e_2,e_3)$ of degree four such that \eqref{eq:11} becomes
\begin{equation}
  \label{eq:12}
  E(e_0,e_1,e_2,e_3) = F(e_1,e_2,e_3) + e_0\;G(e_1,e_2,e_3) = 0.
\end{equation}
Hence, solution planes are characterized by the simultaneous vanishing
of $F$ and $G$. Since these equations are independent of $e_0$, the
value of the homogenizing coordinate can be chosen arbitrarily. This
is the algebraic manifestation of our observation that the solution
consists of pencils of parallel planes.

Since the two homogeneous equations $F$ and $G$ are of respective
degrees four and three there exist at most twelve pencils of parallel
planes whose homologous images are tangent to a cone of
revolution. But not all of them are valid: The circle-cone equation
$G$ also identifies vectors $\vec{n}_0$, $\vec{n}_1$, $\vec{n}_2$,
$\vec{n}_3$ that span a two-dimensional sub-space. The corresponding
planes $\varepsilon_0$, $\varepsilon_1$, $\varepsilon_2$,
$\varepsilon_3$ are all parallel to a fixed direction but do not
qualify as solution planes and have to be singled out. The complete
description of all solution planes is given in

\begin{theorem}
  \label{th:homologous-planes}
  The locus of planes in the moving space whose homologous images are
  tangent to a cone of revolution consists of six pencils of parallel
  planes:
  \begin{enumerate}
  \item The planes orthogonal to one of the four revolute axes
    $r_{01}$, $r_{12}$, $r_{23}$, $r_{30}$ and
  \item the planes orthogonal to one of the two transversals $u$ or
    $v$ of these axes.
  \end{enumerate}
\end{theorem}

\begin{proof}
  We already know that there exist at most twelve parallel pencils of
  solution planes, obtained by solving the quartic equation
  $F(e_1,e_2,e_3) = 0$ and the cubic equation $G(e_1,e_2,e_3) = 0$. By
  Lemma~\ref{lem:spherical-coplanar} below (which actually has been
  proved in a more general context in
  \cite[pp.~131--132]{bottema90:_theoretical_kinematics}), six of
  these pencils of planes are spurious and have to be subtracted from
  the totality of all solution planes. Thus, we only have to show that
  the planes enumerated in the Theorem are really solutions. This is
  obvious in case of a plane $\varepsilon$ orthogonal to axis
  $r_{i,i+1}$ since $\varepsilon_i$ and $\varepsilon_{i+1}$ are
  identical. In case of a plane orthogonal to $u$ or $v$ it is an
  elementary consequence of Theorem~\ref{th:hyperboloid}.
\end{proof}

\begin{lemma}
  \label{lem:spherical-coplanar}
  Consider four spherical displacements $\sigma_0$, $\sigma_1$,
  $\sigma_2$, and $\sigma_3$. Then there exist in general six lines
  (possibly complex or coinciding) through the origin of the moving
  coordinate frame, whose homologous lines are co-planar.
\end{lemma}

\begin{proof}
  The co-planarity condition applied to a vector $\vec{x}$ reads
  \begin{equation}
    \label{eq:13}
    \rank (\vec{x}_0, \vec{x}_1, \vec{x}_2, \vec{x}_3) \le 2.
  \end{equation}
  We compute classes of proportional solution vectors by solving the
  two cubic equations
  \begin{equation}
    \label{eq:14}
    \det (\vec{x}_0, \vec{x}_1, \vec{x}_2) =
    \det (\vec{x}_0, \vec{x}_1, \vec{x}_3) = 0.
  \end{equation}
  It has nine solutions. But the three fixed directions (one real, two
  conjugate imaginary) of the relative rotation $\sigma_2 \circ
  \sigma_1^{-1}$ are spurious so that only six valid solutions remain.
\end{proof}

\section{Homologous lines on a hyperboloid of revolution}
\label{sec:homologous-lines}

Finally, we also investigate the locus of lines $\ell$ in the moving
frame whose homologous images $\ell_0$, $\ell_1$, $\ell_2$, $\ell_3$
form a skew quadrilateral on a hyperboloid of revolution
(Figure~\ref{fig:homologous-points-planes-lines}, right). As usual, we
restrict ourselves to generic configurations. Moreover, we make an
additional assumption on the line's orientation: Setting $L_{i,i+1}
\defeq \ell_i \cap \ell_{i+1}$ (indices modulo four) we require that,
when walking around the edges of the skew quadrilateral $(L_{01},
L_{12}, L_{23}, L_{30})$ we follow the line's orientation only on
every second edge (Figure~\ref{fig:homologous-points-planes-lines},
right).

The manifold of lines in Euclidean three-space depends on four
parameters. The mentioned condition on the four homologous images
$\ell_0$, $\ell_1$, $\ell_2$, $\ell_3$ of a line $\ell$ imposes,
however, five conditions: four intersection conditions for the skew
quadrilateral and one further condition for the hyperboloid of
revolution. Therefore, we expect no solution for a general screw
quadrangle. On the other hand we already know
(Theorem~\ref{th:hyperboloid}) that for rotation quadrangles the
homologous images of two transversals $u$ and $v$ form the required
configuration. We will show that these are the only solutions.

Considering the skew quadrilateral condition only, it is obvious that
either
\begin{itemize}
\item $\ell_i$ and $\ell_{i+1}$ intersect $\alpha_i(r_{i,i+1})$,
  possibly at a point at infinity, or
\item $\ell_i$ and $\ell_{i+1}$ are orthogonal to
  $\alpha_i(r_{i,i+1})$.
\end{itemize}
Only the first condition is compatible with our requirement on the
orientation of the homologous lines. Therefore, it is necessary that
$\ell$ intersects all lines $r_{i,i+1}$ (indices modulo four). Hence
$\ell = u$ or $\ell = v$. Combining this result with
Theorem~\ref{th:hyperboloid} we arrive at

\begin{theorem}
  \label{th:homologous-lines}
  The two transversals $u$ and $v$ of the four relative revolute axes
  $r_{01}$, $r_{12}$, $r_{23}$, and $r_{30}$ are the only lines whose
  homologous images form a skew quadrilateral (with proper
  orientation) on a hyperboloid of revolution.
\end{theorem}

\section{Conclusion and future research}
\label{sec:conclusion}

We investigated geometric structures related to four displacements
$\alpha_i$ ($i \in \{0,1,2,3\}$) such that the relative displacements
$\tau_{i,i+1} = \alpha_{i+1} \circ \alpha_i^{-1}$ are pure
rotations. In particular, we completely characterized the points,
planes and lines whose homologous images lie on a circle, are tangent
to a cone of revolution or form a skew quadrilateral on a hyperboloid
of revolution (with proper orientation). In a next step we plan to
study quadrilateral nets of proper Euclidean displacements such that
neighbouring positions correspond in a pure rotation. Using the
results of this article it is possible to show that independent pairs
of principal contact element nets
\cite{bobenko07:organizing_principles} can occur as trajectories of
such nets. Since mechanically generated motions are often composed of
pure rotations the theory might also be useful in more applied
settings.



%





%
\end{document}